%% file: ginalp.tex
\RequirePackage[l2tabu, orthodox, abort]{nag} 

\documentclass[a4paper, 11pt, pdftex, final]{scrartcl}

\input{preamble}
\begin{document}

\input{topdefs}


\title{Geometric integration of non-autonomous Hamiltonian problems}
\author{Håkon Marthinsen\thanks{Department of Mathematical Sciences, Norwegian University of Science and Technology (NTNU).} \and Brynjulf Owren\footnotemark[1]}

\date{\today}

\maketitle

\begin{abstract}
	\noindent Symplectic integration of autonomous Hamiltonian systems is a well-known field of study in geometric numerical integration, but for non-autonomous systems the situation is less clear, since symplectic structure requires an even number of dimensions.
	We show that one possible extension of symplectic methods in the autonomous setting to the non-autonomous setting is obtained by using canonical transformations.
	Many existing methods fit into this framework.
	We also perform experiments which indicate that for exponential integrators, the canonical and symmetric properties are important for good long time behaviour.
	In particular, the theoretical and numerical results support the well documented fact from the literature that exponential integrators for non-autonomous linear problems have superior accuracy compared to general ODE schemes.
\end{abstract}

\section{Introduction}

\input{introduction.tex}

%
%
%
%
%
%
%
%
%
%
%
%

\section{Four classes of problems}

\input{chap2.tex}

\section{Autonomous and non-autonomous Hamiltonian mechanics}
\label{sec:ham-mech}

In this section, we discuss the dynamics of autonomous (i.e.\ time-independent) and non-autonomous (i.e.\ time-dependent) Hamiltonian systems.

\subsection{Autonomous Hamiltonian systems}
We will first review the basics of autonomous Hamiltonian systems~\parencite{marsden99-1,arnold89}.
Let $Q$ be a smooth $n$-dimen\-sional manifold, and denote its cotangent bundle as $\coT Q$.
The manifold~$Q$ is called the configuration space, and $\coT Q$ is called the phase space.
We will often use $(q, p)$ as an element of $\coT Q$, where $q \in Q$ and $p \in \mathrm{T}^{*}_q Q$.
A Hamiltonian~$H \from \coT Q \to \RR$, together with a symplectic 2-form~$\omega_0$ on $\coT Q$, determine the Hamiltonian vector field~$X_H$ via the equation
\begin{equation} \label{eq:ham-vf}
	\intprod_{X_H} \omega_0 = -\extdiff H,
\end{equation}
where $\extdiff$ and $\intprod$ are the exterior derivative and the interior product, respectively.
In canonical (also called Darboux) coordinates~$(q^i, p_i)$, we can write $\omega_0 = \extdiff p_i \wedge \extdiff q^i$ (with implicit summation over repeated indices), and \Fref{eq:ham-vf} turns into Hamilton's equations,
\[
	\dot q^i = \pdiff{H}{p_i}, \qquad \dot p_i = -\pdiff{H}{q^i}, \qquad \text{for all } 1 \leq i \leq n.
\]
It can be easily proved~\parencite[Section~5.4]{marsden99-1} that the autonomous Hamiltonian~$H$ and the symplectic form~$\omega_0$ are conserved along the integral curves of $X_H$.

Hamiltonian AN problems can be solved numerically by standard symplectic integrators~\parencite[Chapter~VI]{hairer06}, e.g.\ using symplectic, partitioned Runge--Kutta (SPRK) methods.

\subsection{Non-autonomous Hamiltonian systems} \label{sec:non-aut-Ham-sys}

We will now consider then non-autonomous case, i.e.\ when $H$ depends on time as well as phase space, so $H \from \coT Q \times \RR \to \RR$.
The characterization of Hamiltonian vector fields using the symplectic 2-form~\Fref{eq:ham-vf} is no longer appropriate, since $\coT Q \times \RR$ is an odd-dimensional space, while the symplectic 2-form requires an even-dimensional phase space.
Hamilton's equations still apply unchanged, but $H$ is no longer conserved along the integral curves of $X_H$.

Let
\[
	\JJ_n \defeq \begin{bmatrix}
		0      & \II_n \\
		-\II_n & 0
	\end{bmatrix},
\]
where $\II_n$ is the $n \times n$ identity matrix.
Writing $y = (q, p)$ as a column vector, Hamilton's equations in canonical coordinates become
\[
	\dot y = \JJ_n H_y\trans, \qquad H_y = \begin{bmatrix}
		\pdiff{H}{q^1} & \cdots & \pdiff{H}{q^n} & \pdiff{H}{p_1} & \cdots & \pdiff{H}{p_n}
	\end{bmatrix}.
\]

For the case of Hamiltonian NL problems, we need $\dot y = A(t) y$.
Consider a generic Hamiltonian which is quadratic in the phase space variables,
\[
	H = -\tfrac{1}{2} y\trans \JJ_n A(t) y.
\]
We may assume without loss of generality that $A(t) \in \Liesp(2 n)$.
Since $\JJ_n A + A\trans \JJ_n = 0$, it follows that $\JJ_n A$ is symmetric, and we get $\dot y = \JJ_n H_y\trans = A(t) y$.


\subsubsection{Contact structure}
The usual way to tackle this problem is to apply \emph{contact structure}~\parencite[Chapter~5]{abraham78}, \parencite[Appendix~4]{arnold89}.
We use notation similar to \textcite{asorey83}.

Let $\tau \from \coT Q \times \RR \to \coT Q$ be the projection~$(q, p, t) \mapsto (q, p)$, and let $\omega_0 = \extdiff p_i \wedge \extdiff q^i$ be the canonical symplectic form on $\coT Q$, as before.
Define $\tilde \omega_0 \defeq \tau^{*} \omega_0$.
The contact structure on $\coT Q \times \RR$ is then given by the contact form
\[
	\omega_H \defeq \tilde \omega_0 - \extdiff H \wedge \extdiff t.
\]
This enables us to define the (time-dependent) vector field~$X_H$ via
\[
	\intprod_{X_H} \omega_H = 0, \qquad \intprod_{X_H} \extdiff t = 1,
\]
or in canonical coordinates
\begin{equation} \label{eq:contact-Ham}
	\dot q^i = \pdiff{H}{p_i}, \qquad \dot p_i = -\pdiff{H}{q^i}, \qquad \dot t = 1, \qquad \text{for all } 1 \leq i \leq n.
\end{equation}
The contact form~$\omega_H$ is preserved along the flow of $X_H$, but $H$ is not.

\subsubsection{Extended phase space}
An alternative to using contact structure is to append one more dimension to $\coT Q \times \RR$, thus obtaining an even-dimensional \emph{extended phase space}~$\coT(Q \times \RR)$.
Because of this, we may now mimic the autonomous case and define the Hamiltonian system using symplectic forms.
We denote the new variable as $u$.

Let $(q, p, t, u) \in \coT (Q \times \RR)$, and let $\mu \from \coT (Q \times \RR) \to \coT Q \times \RR$ be the projection $(q, p, t, u) \mapsto (q, p, t)$.
We define the \emph{extended Hamiltonian}~$K \from \coT(Q \times \RR) \to \RR$ as
\[
	K \defeq H \circ \mu + u,
\]
and the symplectic form on the extended phase space as
\begin{equation} \label{eq:Omega_0}
	\Omega_0 \defeq \mu^{*} \tilde \omega_0 + \extdiff u \wedge \extdiff t,
\end{equation}
or in canonical coordinates, $\Omega_0 = \extdiff p_i \wedge \extdiff q^i + \extdiff u \wedge \extdiff t$.
The vector field~$X_K$ is then defined the same way as in the autonomous case by
\[
	\intprod_{X_K} \Omega_0 = -\extdiff K,
\]
which in canonical coordinates can be written as
\begin{equation} \label{eq:ext-ham}
	\dot q^i = \pdiff{K}{p_i} =  \pdiff{H}{p_i}, \qquad \dot p_i = -\pdiff{K}{q^i} = -\pdiff{H}{q^i}, \qquad \dot t = \pdiff{K}{u} = 1, \qquad \dot u = -\pdiff{K}{t} = -\pdiff{H}{t},
\end{equation}
for all $1 \leq i \leq n$.
Note that $H$ does not depend on $u$, so we can consider the equation for $\dot u$ as superfluous.
If we disregard the equation for $\dot u$, the equations are the same as for the contact structure approach~\eqref{eq:contact-Ham}.
Thus, the integral curves in extended phase space project (via $\mu$) onto the integral curves defined by the contact structure in $\coT Q \times \RR$.
However, if we want to retain the usual notion of symplecticity of the flow of the vector field, we need to retain the equation for $\dot u$.

Analogous to the autonomous case, both $\Omega_0$ and $K$ are conserved along the flow of $X_K$.
If we choose the initial values~$q_0$, $p_0$, $t_0$, and $u_0 = -H(q_0, p_0, t_0)$, we get that $K = 0$ along the flow.
This allows us to interpret $-u$ as the energy of the Hamiltonian system.

\subsection{Canonical transformations}

It is a well known fact that symplectic integrators for autonomous problems have excellent long-time properties,
however it is not clear whether the same is true for non-autonomous problems.
An enticing thought is to use the constructions from the previous section so that we get a well-defined concept replacing symplecticity for the non-autonomous case.
The solution employed by \textcite{asorey83} is to extend symplectic maps to canonical transformations, as defined below.

\begin{definition} \label{def:can-trans}
	A canonical transformation of a time-dependent system~$(\coT Q \times \RR, \omega_H)$ is a pair~$(\psi, \varphi)$ of diffeomorphisms, $\psi$ on $\coT(Q \times \RR)$ and $\varphi$ on $\coT Q \times \RR$ such that
	\begin{enumerate}
		\item $\mu \circ \psi = \varphi \circ \mu$, and
		\item $\psi^{*} \Omega_0 = \Omega_0$ (i.e.\ $\psi$ is a symplectomorphism).
	\end{enumerate}
\end{definition}
The condition~$\mu \circ \psi = \varphi \circ \mu$ means that the diagram
\[
	\begin{tikzcd}
		\coT (Q \times \RR) \rar{\mu} \dar{\psi} & \coT Q \times \RR \dar{\varphi} \\
		\coT (Q \times \RR) \rar{\mu}            & \coT Q \times \RR
	\end{tikzcd}
\]
commutes.
A consequence of \Fref{def:can-trans} is that $\psi$ must be a symplectomorphism of the form $\psi(q, p, t, u) = \paren*{\varphi(q, p, t), \bar \varphi(q, p, t, u)}$, where $\bar \varphi \from \coT(Q \times \RR) \to \RR$.
We will sometimes refer to $\psi$ as a canonical transformation when there exists a $\varphi$ such that $(\psi, \varphi)$ is a canonical transformation.

Many integrators already exist in $\coT Q \times \RR$, e.g.\ Magnus integrators for linear problems.
Given such an integrator~$\varphi$, we seek a matching $\psi$ such that $(\psi, \varphi)$ is a canonical transformation.
Using ideas similar to those of \textcite{asorey83}, we have the following theorem which characterizes canonical transformations where time is advanced by a constant~$h$.

\begin{theorem} \label{thm:existence}
	Let $\varphi$ be a diffeomorphism of\/ $\coT Q \times \RR$ where the $t$-component is advanced by a constant time-step~$h$.
	Then the following are equivalent.
	\begin{enumerate}[label=(\roman*),ref=\roman*]
		\item \label{it:ex-1} $(\psi, \varphi)$ is a canonical transformation.
		\item \label{it:ex-2} There exists a function~$W \from \coT Q \times \RR \to \RR$ such that
		\begin{equation} \label{eq:omega-tilde}
			\varphi^{*} \tilde \omega_0 = \tilde \omega_0 - \extdiff W \wedge \extdiff t,
		\end{equation}
		and $\psi = (\varphi \circ \mu, u + W \circ \mu)$.
	\end{enumerate}
\end{theorem}

\begin{proof}
	Assume that \eqref{it:ex-1} is true.
	We know that $\psi^{*} \Omega_0 = \Omega_0$.
	Inserting \Fref{eq:Omega_0}, applying $\mu \circ \psi = \varphi \circ \mu$ and $\psi^{*} \extdiff t = \extdiff t$, and rearranging, we get
	\begin{equation} \label{eq:thm-proof-1}
		\mu^{*} (\varphi^{*} \tilde \omega_0 - \tilde \omega_0) = -\extdiff (\psi^{*} u - u) \wedge \extdiff t.
	\end{equation}
	Let~$\nu \from \coT Q \times \RR \to \coT (Q \times \RR)$ be any map such that $\mu \circ \nu = \id$.
	Our candidate function is $W = \nu^{*} (\psi^{*} u - u)$.
	We apply $\nu^{*}$ to both sides of \eqref{eq:thm-proof-1}, insert the candidate function, and get
	\[
		\varphi^{*} \tilde \omega_0 - \tilde \omega_0 = -\extdiff W \wedge \extdiff \nu^{*} t.
	\]
	In the following, we will use the same symbol for the coordinate function for time $t$ in both $\coT (Q \times \RR)$ and $\coT Q \times \RR$.
	Since $\mu^{*} t = t$, we also have that $\nu^{*} t = t$ and we end up with \eqref{eq:omega-tilde}.
	Inserting \eqref{eq:omega-tilde} into \eqref{eq:thm-proof-1}, applying $\mu^{*} t = t$, and rearranging, we obtain
	\begin{equation} \label{eq:f}
		\extdiff (\mu^{*} W - \psi^{*} u + u) \wedge \extdiff t = 0.
	\end{equation}
	Let $f \defeq \mu^{*} W - \psi^{*} u + u$.
	From \eqref{eq:f}, we see that $f$ can only depend on $t$.
	If we apply $\nu^{*}$ to $f$ and insert the candidate function, we see that $f \circ \nu = 0$.
	Since $f$ only depends on $t$, this implies that $f = 0$, proving that $u \circ \psi = u + W \circ \mu$.
	
	Conversely, assume now that \eqref{it:ex-2} is true.
	The map $\psi$ is given by
	\begin{equation} \label{eq:psi}
		\mu \circ \psi = \varphi \circ \mu \qquad \text{together with} \qquad u \circ \psi = u + W \circ \mu.
	\end{equation}
	We apply $\mu^{*}$ and \Fref{eq:psi} to \Fref{eq:omega-tilde} and get
	\[
		\psi^{*} \mu^{*} \tilde \omega_0 = \mu^{*} \tilde \omega_0 - \mu^{*} (\extdiff W \wedge \extdiff t).
	\]
	Since $\mu^{*} t = t$, we get
	\[
		\psi^{*} \mu^{*} \tilde \omega_0 = \mu^{*} \tilde \omega_0 - \extdiff (\psi^{*} u - u) \wedge \extdiff t.
	\]
	Using \Fref{eq:Omega_0}, we obtain
	\[
		\psi^{*} (\Omega_0 - \extdiff u \wedge \extdiff t) = \Omega_0 - \extdiff u \wedge \extdiff t - \extdiff (\psi^{*} u - u) \wedge \extdiff t,
	\]
	or, by applying the fact that $\psi^{*} \extdiff t = \extdiff t$,
	\[
		\psi^{*} \Omega_0 - \Omega_0 = \extdiff \paren*{\psi^{*} u - u - (\psi^{*} u - u)} \wedge \extdiff t = 0.
	\]
\end{proof}

From this point, we will work in canonical coordinates.
This will make the connection with existing numerical methods clearer, as well as provide formulas that can be used directly in numerical calculations.
We will regard $q = (q^i)_{i = 1}^n$, $Q = (Q^i)_{i = 1}^n$, $p = (p^i)_{i = 1}^n$, and $P = (P^i)_{i = 1}^n$ as column vectors.
Let $z = (q, t, p, u)$ and $Z = (Q, t + h, P, U) = \psi(z)$ be column vectors in $\RR^{2 n + 2}$, with $U = u + W(q, p, t)$.
Since $Q$ and $P$ are independent of $u$, the Jacobian matrix~$Z_z \defeq \psi'(z)$ can be written as
\[
	Z_z =
	\begin{bmatrix}
		Q_q & Q_t & Q_p & 0 \\
		0   & 1   & 0   & 0 \\
		P_q & P_t & P_p & 0 \\
		W_q & W_t & W_p & 1
	\end{bmatrix},
	\qquad \text{where} \qquad Q_q \defeq
	\begin{bmatrix}
		\pdiff{Q^1}{q^1} & \cdots & \pdiff{Q^1}{q^n} \\
		\vdots           &        & \vdots           \\
		\pdiff{Q^n}{q^1} & \cdots & \pdiff{Q^n}{q^n}
	\end{bmatrix},
\]
and similarly for the other submatrices.
Let
\[
	Y_y \defeq
	\begin{bmatrix}
		Q_q & Q_p \\
		P_q & P_p
	\end{bmatrix}, \qquad \text{and} \qquad Y_t \defeq
	\begin{bmatrix}
		Q_t \\
		P_t
	\end{bmatrix}.
\]
\begin{proposition} \label{prn:coord}
	In canonical coordinates, condition~\Fref{eq:omega-tilde} in \Fref{thm:existence} is equivalent to $Y_y \in \LieSp(2n)$ together with
	\begin{equation} \label{eq:W-cond}
		W_y \defeq \begin{bmatrix} W_q & W_p \end{bmatrix} = -Y_t\trans \JJ_n Y_y.
	\end{equation}
\end{proposition}

\begin{proof}
	Assume that \Fref{eq:omega-tilde} is satisfied.
	In canonical coordinates, \Fref{eq:omega-tilde} is
	\[
		\extdiff P_i \wedge \extdiff Q^i = \extdiff p_i \wedge \extdiff q^i - \paren[\bigg]{\pdiff{W}{q^i} \extdiff q^i + \pdiff{W}{p_i} \extdiff p_i} \wedge \extdiff t.
	\]
	It is straight-forward to show that this is equivalent to the five equations
	\begin{align}
		\II_n &= P_p\trans Q_q - Q_p\trans P_q, \label{eq:sympl-1} \\
		0     &= P_q\trans Q_q - Q_q\trans P_q, \label{eq:sympl-2} \\
		0     &= P_p\trans Q_p - Q_p\trans P_p, \label{eq:sympl-3} \\
		W_q   &= P_t\trans Q_q - Q_t\trans P_q, \label{eq:sympl-4} \\
		W_p   &= P_t\trans Q_p - Q_t\trans P_p. \label{eq:sympl-5}
	\end{align}
	Equations~\Fref{eq:sympl-1}--\Fref{eq:sympl-3} may be written as
	\begin{equation*} 
		Y_y\trans \JJ_n Y_y = \JJ_n,
	\end{equation*}
	which implies $Y_y \in \LieSp(2 n)$.
	Using $U = u + W$, we can write conditions~\Fref{eq:sympl-4}--\Fref{eq:sympl-5} as
	\begin{equation*}
		W_y = -Y_t\trans \JJ_n Y_y.
	\end{equation*}
	To prove the converse, simply reverse the proof.
\end{proof}

In the autonomous setting, canonical transformations are equivalent to symplectic maps.
To see this, assume that we are in the autonomous setting, and are given a symplectic map~$Y = \varphi(y)$.
Then $Y_y \in \LieSp(2 n)$, $Y_t = 0$, and \Fref{eq:W-cond} is satisfied by, say, $W = 0$, giving $U = u$.
Thus, $\varphi$ can be turned into a canonical transformation simply by appending the trivial update equation~$U = u$.
This is compatible with the earlier observation that $-u$ may be regarded as the energy of the system.

\section[Canonical transf.\ and integrators for non-auton.\ Hamiltonian systems]{Canonical transformations and integrators for non-autonomous Hamiltonian systems}

In this section we take a look at how existing methods with constant time-step fit into the framework of canonical transformations.
We consider the two situations where we are given either $\varphi \from \coT Q \times \RR \to \coT Q \times \RR$ or $\psi \from \coT (Q \times \RR) \to \coT (Q \times \RR)$, and would like to find the complementing map such that $(\psi, \varphi)$ is a canonical transformation.

\subsection{Constructing a canonical transformation from a given map~\texorpdfstring{$\varphi$}{φ}}

In general, this situation is already covered by \Fref{thm:existence}.

\begin{corollary} \label{cor:phi-to-psi-Lie}
	Let $(Y, T) = \varphi(y, t) = \paren*{M(t) y, t + h}$, where $M(t) \in \LieSp(2 n)$.
	Then $(\psi, \varphi)$ is a canonical transformation, with $\psi = (\varphi \circ \mu, u + W \circ \mu)$ and
	\begin{equation} \label{eq:W-lie}
		W = \tfrac{1}{2} y\trans M(t)\trans \JJ_n M'(t) y.
	\end{equation}
\end{corollary}

\begin{proof}
	From \Fref{prn:coord}, we need that $Y_y = M(t) \in \LieSp(2 n)$, which is clearly satisfied.
	Condition~\Fref{eq:W-cond} says that we must find a $W$ such that
	\[
		W_y = -Y_t\trans \JJ_n Y_y = -y\trans M'(t)\trans \JJ_n M(t).
	\]
	Integrating with respect to $y$ and transposing the result, we obtain \Fref{eq:W-lie}.
	Thus, by \Fref{thm:existence}, $(\psi, \varphi)$ is a canonical transformation.
\end{proof}

The class of methods~$\varphi$ in \Fref{cor:phi-to-psi-Lie} contains among others, Magnus methods~\parencite{iserles99} $M(t) = \exp\paren[\big]{h X(t)}$, Fer methods~\parencite{iserles84} and commutator-free methods~\parencite{celledoni03-1} $M(t) = \prod_i \exp\paren[\big]{h X_i(t)}$, and Cayley methods~\parencite{marthinsen01} $M(t) = \cay\paren[\big]{h X(t)}$, where $X(t)$ and $X_i(t)$ are elements of $\Liesp(2 n)$.
All of these can be applied to non-autonomous linear Hamiltonian problems~$\dot y = A(t) y$ (i.e.\ NL problems).
To get consistent methods, we must choose $M$ carefully.
In fact, by considering the modified vector field of the methods, we get that the methods of this class are consistent if $M(t)\rvert_{h = 0} = \II_{2 n}$, and
\[
	\eval*{\diff{M(t)}{h}}_{h = 0} = A(t).
\]

\begin{example} \label{exa:mag-exp}
	Magnus integrators fit into this format by choosing $M(t) = \exp\paren[\big]{h X(t)}$ for $X \from \RR \to \Liesp(2 n)$.
	Consistency requires
	\[
		\eval*{X(t)}_{h = 0} = A(t).
	\]
	Since $M'(t) = h \paren[\big]{\dexp_{h X(t)} X'(t)} M(t)$, and $M(t)\trans \JJ_n M(t) = \JJ_n$, we can apply \Fref{cor:phi-to-psi-Lie} and express $U = u + W \circ \mu$ as
	\[
		U = u + \frac{h}{2} y\trans \JJ_n \paren*{\dexp_{-h X(t)} X'(t)} y,
	\]
	or alternatively as
	\[
		U = u + \frac{h}{2} Y\trans \JJ_n \paren*{\dexp_{h X(t)} X'(t)} Y.
	\]
\end{example}

\subsection{Constructing a canonical transformation from a given map~\texorpdfstring{$\psi$}{ψ}}

Assume that we are given a symplectomorphism~$\psi \from \coT(Q \times \RR) \to \coT(Q \times \RR)$, where the $t$-component is advanced by a constant time-step~$h$, i.e.\ $\psi$ is the map~$(q, t, p, u) = z \mapsto Z = (Q, t + h, P, U)$.
We seek a map~$\varphi$ such that $(\psi, \varphi)$ is a canonical transformation.
This is only possible if $Q$ and $P$ are independent of $u$, since we need $\mu \circ \psi = \varphi \circ \mu$.

In the following proposition, we will use local coordinates and write $x = (q, p, t) \in \coT Q \times \RR$, and $\kappa \from (q, p, u) \mapsto (q, p, h, u)$.
\begin{proposition} \label{prn:psi-to-phi}
	Let $H \from \coT Q \times \RR \to \RR$ be a Hamiltonian.
	Any symplectomorphism~$Z = \psi(z)$ which can be expressed in coordinates as
	\[
		Z = z + \kappa \circ F(H_x) \circ \mu(z),
	\]
	where $F \from \Diff(\RR^{2 n + 1}) \to \Diff(\RR^{2 n + 1})$, is a canonical transformation.
\end{proposition}

\begin{proof}
	The only way we can satisfy $\mu \circ \psi = \varphi \circ \mu$ is if both $Q$ and $P$ are independent of $u$.
	The vector~$H_x$ consists of the partial derivatives of $H$, which is independent of $u$, so $F(H_x)$ also has to be independent of $u$.
	Thus, the only component of $Z$ that can depend on $u$ is $U$, proving that there exists a $\varphi$ satisfying $\mu \circ \psi = \varphi \circ \mu$.
\end{proof}

\begin{corollary}
	Symplectic partitioned Runge--Kutta (SPRK) methods applied to the Hamiltonian problem with extended Hamiltonian~$K = u + H \circ \mu$ are canonical transformations.
\end{corollary}

\begin{proof}
	This follows immediately from \Fref{prn:psi-to-phi}.
\end{proof}

\begin{example} \label{exa:SPRK}
	Let us check that SPRK methods actually are canonical transformations by calculating $W$ and $\varphi$.
	
	An SPRK method is given by a Butcher tableau with coefficients~$a_{i, j}$ and $b_i \neq 0$.
	The second Butcher tableau (marked by a hat) in the partitioned method is given by the first one via the formulas $\hat a_{i, j} = b_j - a_{j, i} b_j / b_i$ and $\hat b_i = b_i$.
	The method will then be
	\begin{gather*}
		\bar k_i = \pdiff{K}{\bar p} (\bar Q_i, \bar P_i), \qquad \bar l_i = -\pdiff{K}{\bar q} (\bar Q_i, \bar P_i), \\
		\bar Q_i = \bar q + h \sum_{j = 1}^s a_{i, j} \bar k_j, \qquad \bar P_i = \bar p + h \sum_{j = 1}^s \hat a_{i, j} \bar l_j, \\
		\bar Q = \bar q + h \sum_{i = 1}^s b_i \bar k_i, \qquad \bar P = \bar p + h \sum_{i = 1}^{s} b_i \bar l_i. \\
	\end{gather*}
	We can rewrite this using
	\begin{gather*}
		\bar k_i \defeq (k_i, \hat k_i) = \paren*{\pdiff{H}{p} (Q_i, P_i, T_i), 1}, \\
		\bar l_i \defeq (l_i, \hat l_i) = \paren*{-\pdiff{H}{q} (Q_i, P_i, T_i), -\pdiff{H}{t} (Q_i, P_i, T_i)},
	\end{gather*}
	and we obtain
	\begin{gather*}
		Q_i = q + h \sum_{j = 1}^{s} a_{i, j} k_j, \qquad P_i = p + h \sum_{j = 1}^{s} \hat a_{i, j} l_j, \qquad T_i = t + c_i h, \\
		Q = q + h \sum_{i = 1}^{s} b_i k_i, \qquad P = p + h \sum_{i = 1}^{s} b_i l_i, \\
		T = t + h, \qquad U = u + h \sum_{i = 1}^{s} b_i \hat l_i,
	\end{gather*}
	where $c_i = \sum_{j = 1}^{s} a_{i, j}$.
	From these formulas, we see that
	\[
		W = h \sum_{i = 1}^{s} b_i \hat l_i,
	\]
	and $Q$, $P$, $T$ and $W$ are indeed independent of $u$.
	Thus, we have found $\varphi$.

	For non-autonomous, linear problems~$\dot y = A(t) y$, $y = (q, p)$, we have the Hamiltonian~$H = -\frac{1}{2} y\trans \JJ_n A(t) y$, so writing $Y_i = (Q_i, P_i)$, we obtain
	\[
		\begin{bmatrix}
				k_i \\
				l_i
			\end{bmatrix}
			= A(T_i) Y_i, \qquad \hat l_i = \tfrac{1}{2} Y_i\trans \JJ_n A'(T_i) Y_i.
	\]
\end{example}

We note in passing that backward error analysis can be trivially adapted to the situation of non-autonomous Hamiltonian systems.
%
%
%
%
%
%
%
Since we are applying a symplectic method~$\psi$ to a Hamiltonian ODE~$\dot z = f(z)$ in extended phase space with Hamiltonian~$K$, we can apply the result from \parencite[Theorem~IX.3.1]{hairer06}, showing that the modified equation is also Hamiltonian, and has Hamiltonian
\[
	\tilde K(z) = K(z) + h K_2(z) + h^2 K_3(z) + \dotsb.
\]
The differential equation~$\dot t = 1$ is integrated exactly by the numerical method, so we can write $\tilde K = u + \tilde H \circ \mu$.
Thus,
\[
	(\tilde H - H) \circ \mu = h K_2 + h^2 K_3 + \dotsb,
\]
which shows that the energy error for the numerical method is of order no less than the order of the symplectic method in extended phase space.
%
%
%
%

\section{Numerical experiments}

In the numerical experiments, we would like to consider the situation where the non-autono\-mous problem can be viewed as a small perturbation of an autonomous problem with bounded energy.
Other more challenging problems, such as the Airy equation (which has unbounded energy), will not be considered here.
Consider the time-dependent harmonic oscillator with Hamiltonian
\begin{equation} \label{eq:H}
	H(q, p, t) = \frac{1}{2} \paren[\Big]{\paren[\big]{1 + \epsilon \sin (\alpha t)} q\trans q + p\trans p},
\end{equation}
where $q, p \in \RR^n$, $0 < \epsilon \ll 1$, and $0 < \alpha \ll 1$.
As we saw in \Fref{sec:non-aut-Ham-sys}, this Hamiltonian corresponds to the linear ODE
\[
	\dot y = A(t) y, \qquad A(t) =
	\begin{bmatrix}
		0                                                & \II_n \\
		-\paren[\big]{1 + \epsilon \sin(\alpha t)} \II_n & 0
	\end{bmatrix}.
\]
We can think of this oscillator as a slowly varying perturbation of the usual harmonic oscillator.
The time-dependent perturbation ensures that the energy~$H$ and the symplectic 2-form~$\omega_0$ of the system are no longer conserved, but since the perturbation is small and periodic, we expect that the energy is bounded as long as there is no resonance.
%

\subsection{Long-time performance}

We believe that canonical methods may be well suited for non-autonomous Hamiltonian problems where we seek a long-time numerical solution with qualitatively good results.
We will investigate this by considering the symmetric, 4th order Magnus method based on two-stage Gauss--Legendre quadrature~\parencite[Example~IV.7.4]{hairer06}.
We will call this method the Lie--Gauss method.
The update map in $\coT Q \times \RR$ is
\begin{equation} \label{eq:Lie-Gauss}
	(Y, T) = \varphi(y, t) = \paren[\Bigg]{\exp\paren[\bigg]{\frac{h}{2}(A_1 + A_2) + \frac{\sqrt 3 h^2}{12} [A_2, A_1]} y, t + h}, \qquad A_i = A(t + c_i h),
\end{equation}
where $c_1 = \frac{1}{2} - \frac{\sqrt 3}{6}$ and $c_2 = \frac{1}{2} + \frac{\sqrt 3}{6}$.
To obtain a canonical method, we follow the construction from \Fref{cor:phi-to-psi-Lie}.
This gives us the auxiliary update map $u \mapsto U$ as given by \Fref{exa:mag-exp}.
We use initial values $q_0 = (1, 2, 3, 4)$, $p_0 = (4, 1, 2, 3)$, $t_0 = 0$, $u_0 = -H(q_0, p_0, t_0)$, parameters $\alpha = 0.1$ and $\epsilon = 0.3$, and step-length $h = 0.3$.

In order to evaluate the accuracy of the method, we generate a reference solution using the same method, but with a step-length~$h = 0.03$.
Since the method is fourth-order, the reference solution is a much more accurate solution than the other one.
Denote the Hamiltonian evaluated in the reference solution and the approximate solution by $H_\textrm{ex}$ and $H_k$, respectively.
Furthermore, let $u_k$ be the $u$-component of the approximate solution, and let $K_k = u_k + H_k$ be the extended Hamiltonian evaluated in the approximate solution.

\begin{figure}
	\centering
	\subfloat[$H_k - H_\textrm{ex}$ (blue) and $-u_k - H_\textrm{ex}$ (red)\label{fig:gauss2a}]{
		\includegraphics{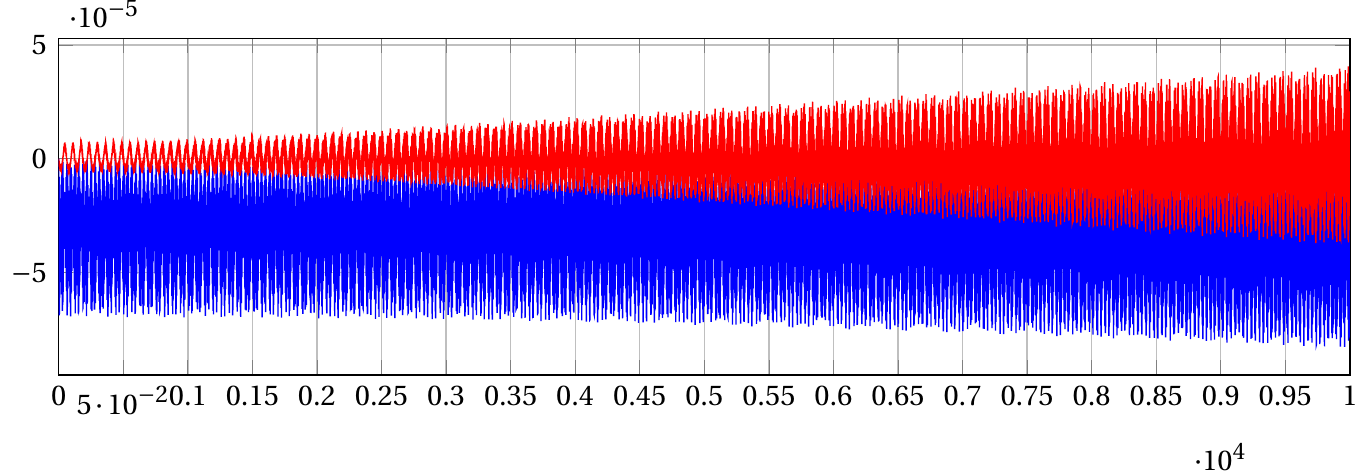}
	} \\
	\subfloat[Close-up of $H_k$ (blue), $-u_k$ (red), and $H_\textrm{ex}$ (brown)\label{fig:gauss2c}]{
		\includegraphics{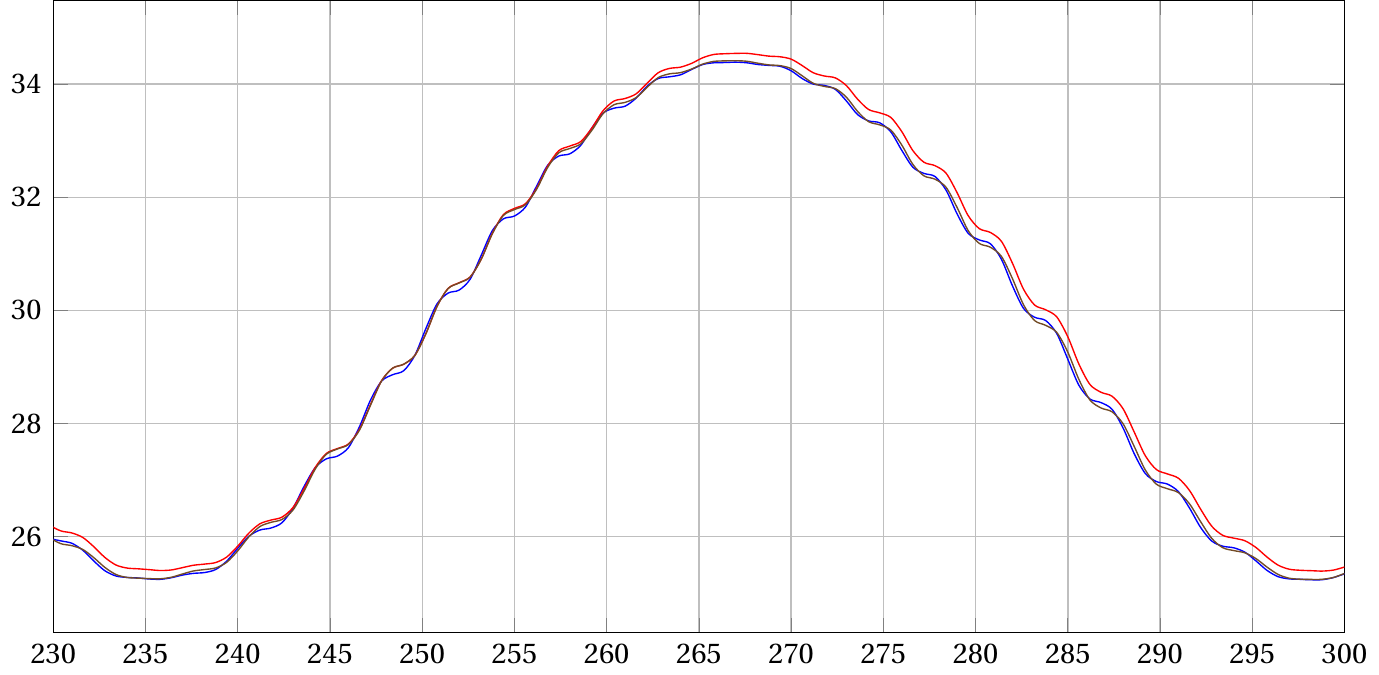}
	}
	\caption{Long-time behaviour.
	Time~$t_k$ is plotted along the $x$-axis}
\end{figure}


In \Fref{fig:gauss2a}, we display $H_k - H_\mathrm{ex}$ and $-u_k - H_\mathrm{ex}$ as functions of time $t_k$.
We see that both $H_k$ and $-u_k$ stay close to $H_\mathrm{ex}$ over long time.
If we had plotted $K_k$, we would have observed that this is well preserved over long time, as expected of canonical methods.
In \Fref{fig:gauss2c}, we plot $H_k$ and $-u_k$ together with $H_\mathrm{ex}$ to get a better understanding of our simulation.
In order to get a visible separation of the three curves, we had to switch to a step-length~$h = 0.6$ together with a lower-order method, namely the 1st~order Magnus method based on the update~$y \mapsto \exp\paren[\big]{h A(t)} y$ (i.e.\ the Lie--Euler method).
We observe that the two approximates to the energy, $H_k$ and $-u_k$, oscillate near the reference solution.

\subsection{Symmetric methods and canonical transformations}

In the previous subsection, we applied a symmetric and canonical method to a non-autono\-mous Hamiltonian problem, and observed good long-time behaviour.
In this subsection, we test the four different combinations of symmetric and canonical methods on the same problem.
Three of the methods, namely methods (a), (b), and (d) below are 2nd order Runge--Kutta methods applied to the Hamiltonian equations~\Fref{eq:ext-ham} in extended phase space.
Method (c) is different and is explained in detail below.
For each Runge--Kutta method, we indicate whether the method is symmetric and/or canonical.
We choose (b) and (d) so that neither of them are conjugate to symplectic in order to rule out this potential source of unwanted good long-time behaviour~\parencite{celledoni13}.
We choose the following methods:
\begin{enumerate}[label=(\alph*),ref=\alph*]
	\item The midpoint method (both symmetric and canonical)
	\item Kahan's method (only symmetric)
	\item A projection-based method (only canonical)
	\item Lobatto IIIC (neither symmetric nor canonical)
\end{enumerate}
The midpoint method is canonical since it can be viewed as an SPRK method (see \Fref{exa:SPRK}).
The projection-based method is based on the idea of projecting the truncated Taylor series of the exact solution onto the symplectic Lie algebra so that we obtain an integrator that can be extended to a canonical transformation.
This demonstrates that we can get canonical methods (and good long-time behaviour) even if we use projections.
Using projection as a device for energy preservation is known to give unsatisfactory results in many cases~(see \parencite[pp.~112--113]{hairer06}).

The exact solution of $\dot y = A(t) y$ is $y(t + h) = \paren[\big]{\II_{2 n} + h A(t) + \OO(h^2)} y(t)$.
Our goal is to use the Taylor series to obtain a consistent method of the format $Y = M(t) y = \exp\paren[\big]{h X(t)} y$ with $X(t) \in \Liesp(2 n)$, as discussed in \Fref{exa:mag-exp}.
Let $\Pi \from \Liegl(2 n) \to \Liesp(2 n)$ be the linear projection
\[
	\Pi(X) = \tfrac{1}{2} (X + \JJ_n X\trans \JJ_n).
\]
The projection-based method is then defined as
\begin{equation} \label{eq:proj}
	Y = M(t) y = \exp \circ \Pi \circ \log\paren[\big]{\II_{2 n} + h A(t)} y,
\end{equation}
together with the auxiliary update equation
\[
	U = u + \frac{h}{2} Y\trans \JJ_n \paren*{\dexp_{h X(t)} X'(t)} Y,
\]
with $X(t) = \tfrac{1}{h} \Pi \circ \log\paren[\big]{\II_{2 n} + h A(t)}$.
By Taylor expansion of the logarithm, we see that $X(t)\rvert_{h = 0} = A(t)$.
Thus, the method is consistent, i.e.\ of order one.

We use initial values $q_0 = (1, 2, 3, 4)$, $p_0 = (4, 1, 2, 3)$, $t_0 = 0$, $u_0 = -H(q_0, p_0, t_0)$, parameters $\alpha = 0.123$ and $\epsilon = 0.6$, and step-length $h = 0.3$.
The time evolution of the Hamiltonian evaluated in the numerical solution, as well as minus the auxiliary variable~$u_k$ are shown in~\Fref{fig:sym-can}.
We observe that all the methods perform well, except for Lobatto IIIC.

\begin{figure}
	\centering
	\subfloat[Midpoint method (both symmetric and canonical)]{
		\includegraphics{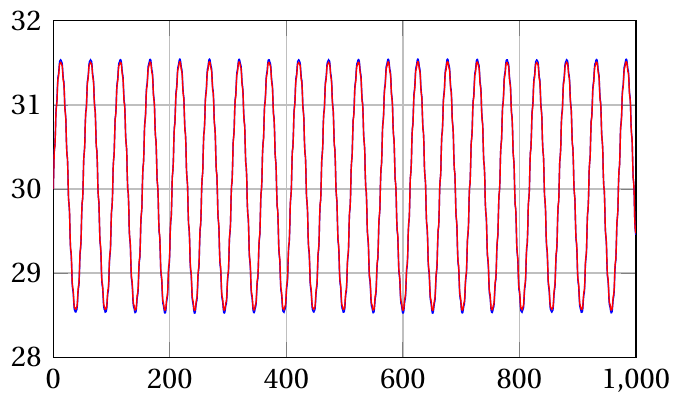}
	} \quad
	\subfloat[Kahan's method (only symmetric)]{
		\includegraphics{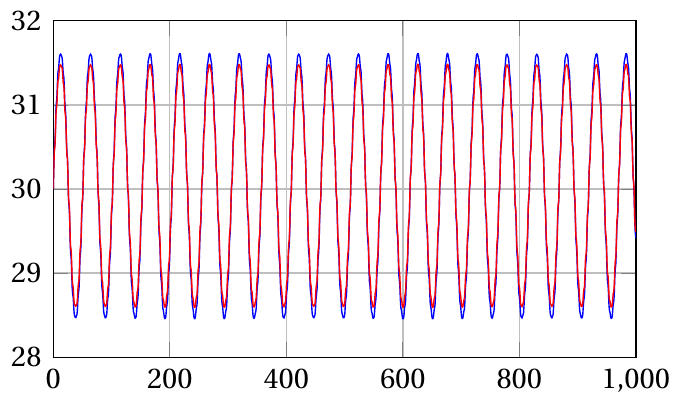}
	} \\
	\subfloat[Projection-based method (only canonical)]{
		\includegraphics{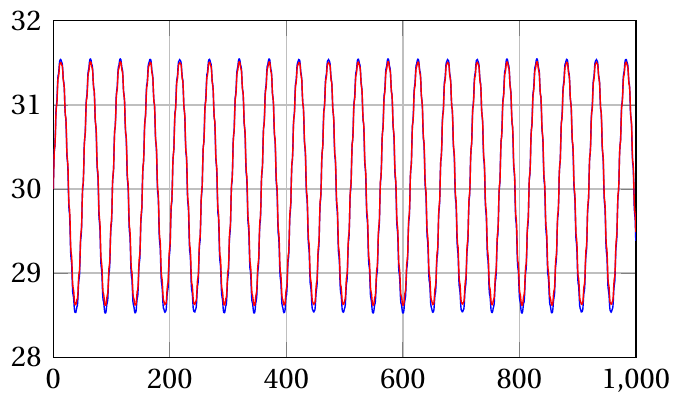}
	} \quad
	\subfloat[Lobatto IIIC (neither symmetric nor canonical)]{
		\includegraphics{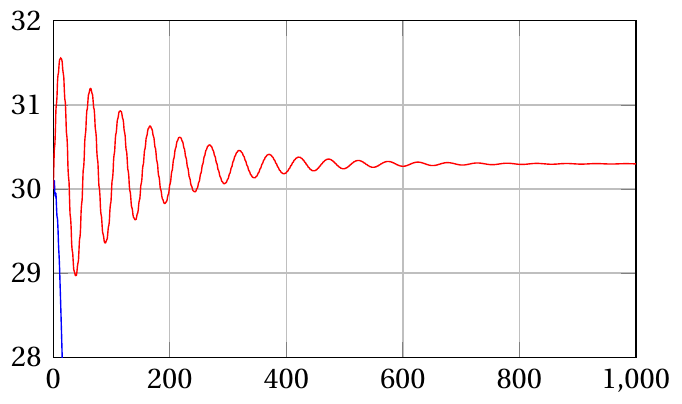}
	}
	\caption{Symmetric and canonical methods.
	$H(q_k, p_k, t_k)$ is plotted in blue and $-u_k$ in red, both with $t_k$ along the $x$-axis
	\label{fig:sym-can}}
\end{figure}

\subsection{Canonical, symmetric, and exponential methods}

In the final experiment, we compare methods with combinations of three different properties, namely canonical, symmetric, and exponential methods.
We have met canonical and symmetric methods earlier, but not exponential methods. By exponential, we mean methods that solve the ODE exactly if they are applied to an autonomous, linear (AL) problem, i.e.\ if $A(t)$ is actually independent of $t$.
Magnus methods are exponential, since all their commutators will disappear, leaving the exact solution in the AL~case.

The methods tested are:
\begin{description}
	\item[Lie--Gauss]
	The fourth order Lie--Gauss method given by \eqref{eq:Lie-Gauss}.
	\item[Lie--midpoint]
	The method given by
	\[
		Y = \exp(h A_{1 / 2}) y, \qquad A_{1 / 2} = A(t + h / 2).
	\]
	\item[Lie--Euler]
	The method given by
	\[
		Y = \exp(h A_0) y, \qquad A_0 = A(t).
	\]
	\item[Gauss--Legendre]
	The fourth order Gauss--Legendre Runge--Kutta method.
	\item[Midpoint]
	The standard midpoint Runge--Kutta method.
	\item[Kahan]
	Kahan's method (viewed as a Runge--Kutta method~\parencite{celledoni13}).
	\item[Projection]
	The projection-based method given by \eqref{eq:proj}.
	\item[Radau IIA]
	The Radau IIA Runge--Kutta method of order three \parencite[Table~IV.5.5]{hairer96}.
	This method was chosen as an example of a method which is neither exponential, symmetric, nor canonical.
	\item[Symplectic Euler]
	The symplectic partitioned Runge--Kutta method \parencite[Theorem~VI.3.3]{hairer06}
	\[
		Q = q + h H_p(P, q, t), \qquad P = p - h H_q(P, q, t).
	\]
	\item[ExpNonCan]
	The method given by
	\[
		Y = \exp \paren[\big]{h A_0 (\II_{2 n} + h A_0 [A_0, A_1])} y, \qquad A_i = A(t + i h).
	\]
	This method has been constructed to be exponential, but not canonical, since we apply the exponential map to something which lies outside of $\Liesp(2 n)$.
	The commutator ensures that we get the exact solution if we apply the method to an AL~problem.
	\item[ExpSymNonCan]
	The method given by
	\[
		Y = \exp \paren[\big]{h A_{1 / 2} (\II_{2 n} + h A_{1 / 2} [A_0, A_1])} y, \qquad A_i = A(t + i h).
	\]
	This method is similar to ExpNonCan, but has been modified to ensure that it is symmetric.
\end{description}
All the methods advance time using $T = t + h$.
We ignore the $u$-component of the canonical methods, since in this experiment we are measuring the energy error~$\lvert H_k - H_\textrm{ex} \rvert$, which is independent of $u$.
In addition to these methods, we also include some compositions of symmetric methods using the triple jump of order 4~\parencite[Example~II.4.2]{hairer06}.
See \Fref{tab:max-energy-errors} for a summary of the properties and order of each of the methods.

We test these methods on the same Hamiltonian as before, with initial values $q_0 = (1, 2, 3, 4)$, $p_0 = (4, 1, 2, 3)$, $t_0 = 0$, $u_0 = -H(q_0, p_0, t_0)$, parameters $\alpha = 0.123$ and $\epsilon = 0.1$, and step-length $h = 0.3$.
The reference solution, giving $H_\textrm{ex}$, is calculated using the fourth order Lie--Gauss method with $h = 0.02$.
The time interval of the experiment is $[0, 50\,000]$.

\begin{figure}
	\centering
	\includegraphics{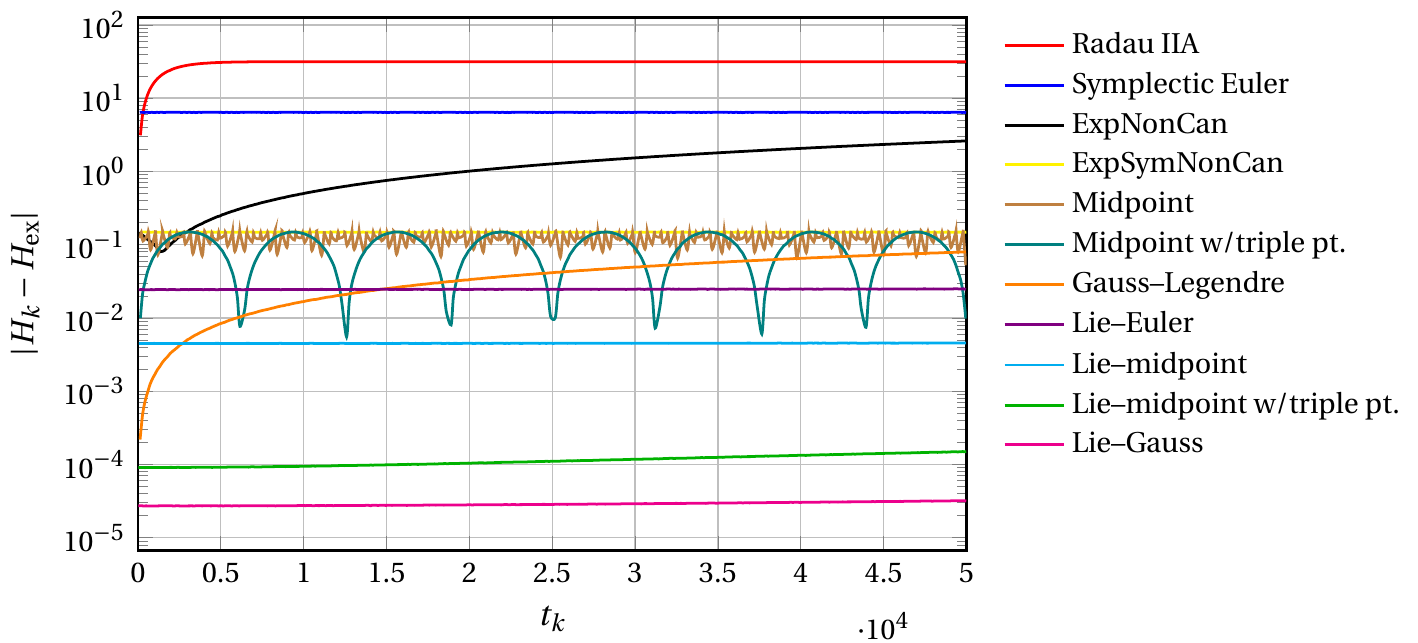}
	\caption{Smoothed energy errors
	\label{fig:long-time-energy-error}}
\end{figure}

The energy error~$\lvert H_k - H_\textrm{ex} \rvert$ oscillates rapidly around zero, and therefore, plotting this quantity is not helpful.
Instead, we divide the time interval into subintervals containing 500~samples each, and plot the maximum energy error within each subinterval.
This procedure smooths out the oscillations, but retains the relevant information about the size of the energy error.
The smoothed energy error is presented in \Fref{fig:long-time-energy-error}.
Many of the schemes yield very similar results, and therefore we only plot some of them.
In particular, the midpoint and Kahan methods give similar results for this example.

In \Fref{tab:max-energy-errors}, we summarize the results of this experiment.
The top part of the table consists of the methods with the smallest energy errors.
They all have maximum energy errors of less than $\approx 0.15$, which is the maximum possible error we can get if the rapidly oscillating component of $H_k$ (see \Fref{fig:gauss2c}) is completely out of phase with the exact solution.
We will call this maximum phase error.
The middle part of the table consists of the methods which follow the slowly oscillating component of $H_k$ fairly well, but which attain the maximum possible phase error of $\approx 0.15$.
The last part of the table contains the worst methods, with errors larger than the maximum possible phase error.

From the table, we see that for this problem, the best methods are the ones which are canonical, symmetric, and exponential.
The methods that perform the worst only have one or none of these properties.
Even though the Gauss--Legendre method is placed in the top tier of the table, we observe in \Fref{fig:long-time-energy-error} that the energy error keeps growing for the whole time interval.
The other methods in this part of the table have energy errors that remain at the same level throughout the interval.

%

\begin{table}
	\caption{Order, properties, and maximum energy errors for a selection of different methods.
	Properties: C~(canonical), S~(symmetric), E~(exponential).\label{tab:max-energy-errors}}
	\medskip
	\centering
	\begin{tabular}{@{}llll@{}}
		\toprule
		Method                         & Order & Properties    & Max energy error     \\ \midrule
		Lie--Gauss                     & 4     & CSE           & $3.20 \cdot 10^{-5}$ \\
		Lie--midpoint with triple jump & 4     & CSE           & $1.50 \cdot 10^{-4}$ \\
		Lie--midpoint                  & 2     & CSE           & $4.56 \cdot 10^{-3}$ \\
		Lie--Euler                     & 1     & C\phantom{S}E & $2.50 \cdot 10^{-2}$ \\
		Gauss--Legendre                & 4     & CS            & $7.98 \cdot 10^{-2}$ \\ \midrule
		Midpoint with triple jump      & 4     & CS            & $1.49 \cdot 10^{-1}$ \\
		Midpoint                       & 2     & CS            & $1.49 \cdot 10^{-1}$ \\
		ExpSymNonCan                   & 1     & \phantom{C}SE & $1.49 \cdot 10^{-1}$ \\
		Kahan with triple jump         & 4     & \phantom{C}S  & $1.50 \cdot 10^{-1}$ \\
		Projection                     & 1     & C             & $1.53 \cdot 10^{-1}$ \\
		Kahan                          & 2     & \phantom{C}S  & $1.68 \cdot 10^{-1}$ \\ \midrule
		ExpNonCan                      & 1     & \phantom{CS}E & $2.61 \cdot 10^{0}$  \\
		Symplectic Euler               & 1     & C             & $6.44 \cdot 10^{0}$  \\
		Radau IIA                      & 3     &               & $3.15 \cdot 10^{1}$  \\ \bottomrule
	\end{tabular}
\end{table}

\section{Conclusion}

\input{conclusion.tex}

\printbibliography

\end{document}

%% file: preamble.tex
\usepackage{fixltx2e} 
\ObsoleteEnv{displaymath}{LaTeX's equation*} 

\usepackage[utf8]{inputenc} 
\usepackage[T1]{fontenc} 

\usepackage[UKenglish]{babel} 
\usepackage[strict]{csquotes}

\usepackage[firstinits,uniquename=false,style=numeric-comp,backend=bibtex,isbn=false]{biblatex}
\defbibheading{bibliography}[\refname]{\section*{#1}} 
\appto{\bibsetup}{\emergencystretch=2.6em} 

\usepackage[final]{microtype} 
\usepackage{scrpage2}
\usepackage{amsmath, amsfonts, amssymb} 

\usepackage{amsthm} 
\makeatletter
\renewenvironment{proof}[1][\proofname]{\par
  \pushQED{\qed}%
  \normalfont \topsep6\p@\@plus6\p@\relax
  \trivlist
  \item[\hskip\labelsep
        \sffamily\bfseries
    #1\@addpunct{.}]\ignorespaces
}{%
  \popQED\endtrivlist\@endpefalse
}
\makeatother

\usepackage{mathtools} 
\mathtoolsset{mathic=true} 

\usepackage[final]{graphicx} 
\usepackage[svgnames]{xcolor} 

\usepackage{tikz} 
\usetikzlibrary{matrix,arrows}

\usepackage{subfig} 

\usepackage{url} 

\usepackage{datetime} 
\usepackage{booktabs} 
\usepackage{colortbl}
\usepackage{todonotes} 
\makeatletter
\renewcommand{\todo}[2][]{\tikzexternaldisable\@todo[#1]{#2}\tikzexternalenable}

\makeatother
\usepackage{enumitem}

\usepackage{fourier}

\DeclareMathSymbol{\partial}{\mathord}{otherletters}{64}

\makeatletter
\def\resetMathstrut@{%
  \setbox\z@\hbox{%
    \mathchardef\@tempa\mathcode`\(\relax
    \def\@tempb##1"##2##3{\the\textfont"##3\char"}%
    \expandafter\@tempb\meaning\@tempa \relax
  }%
  \ht\Mathstrutbox@1.2\ht\z@ \dp\Mathstrutbox@1.2\dp\z@
}
\makeatother

\newdimen\arrowsize
\newdimen\unit
\pgfarrowsdeclare{fourier}{fourier}{\pgfarrowsrightextend{\pgflinewidth}}
{
    \unit=\pgflinewidth
    \divide\unit by 12
    \pgfpathmoveto{\pgfpoint{12\unit}{3\unit}}
    \pgfpathcurveto{\pgfpoint{0\unit}{7\unit}}{\pgfpoint{-18\unit}{19\unit}}{\pgfpoint{-26\unit}{26\unit}}
    \pgfpathlineto{\pgfpoint{-30\unit}{21\unit}}
    \pgfpathcurveto{\pgfpoint{-25\unit}{16\unit}}{\pgfpoint{-20\unit}{9\unit}}{\pgfpoint{-18\unit}{6\unit}}
    \pgfpathlineto{\pgfpoint{-18\unit}{-6\unit}}
    \pgfpathcurveto{\pgfpoint{-20\unit}{-9\unit}}{\pgfpoint{-25\unit}{-16\unit}}{\pgfpoint{-30\unit}{-21\unit}}
    \pgfpathlineto{\pgfpoint{-26\unit}{-26\unit}}
    \pgfpathcurveto{\pgfpoint{-18\unit}{-19\unit}}{\pgfpoint{0\unit}{-7\unit}}{\pgfpoint{12\unit}{-3\unit}}
    \pgfpathclose
    \pgfusepathqfill
}

\usepackage{tikz-cd} 
\tikzset{
	commutative diagrams/.cd,
	arrow style=tikz,
	diagrams={>=fourier}
}

\usepackage[scale=0.86]{tgadventor} 
\usepackage[scaled=0.84]{beramono} 

\usepackage{bm} 

\DeclareMathAlphabet{\mathsfit}{\encodingdefault}{\sfdefault}{m}{sl}
\SetMathAlphabet{\mathsfit}{bold}{\encodingdefault}{\sfdefault}{bx}{sl}

\newcommand{\theoremname}{Theorem}
\newcommand{\propositionname}{Proposition}
\newcommand{\corollaryname}{Corollary}
\newcommand{\lemmaname}{Lemma}
\newcommand{\definitionname}{Definition}
\newcommand{\remarkname}{Remark}
\newcommand{\exercisename}{Exercise}
\newcommand{\examplename}{Example}

\usepackage[plain]{fancyref} 
\newcommand*{\fancyrefthmlabelprefix}{thm}
\newcommand*{\fancyrefprnlabelprefix}{prn}
\newcommand*{\fancyrefcorlabelprefix}{cor}
\newcommand*{\fancyreflemlabelprefix}{lem}
\newcommand*{\fancyrefdeflabelprefix}{def}
\newcommand*{\fancyrefremlabelprefix}{rem}
\newcommand*{\fancyrefexelabelprefix}{exe}
\newcommand*{\fancyrefexalabelprefix}{exa}

\fancyrefaddcaptions{english}{%
    \newcommand*{\Frefthmname}{\theoremname}%
    \newcommand*{\Frefprnname}{\propositionname}%
    \newcommand*{\Frefcorname}{\corollaryname}%
    \newcommand*{\Freflemname}{\lemmaname}%
    \newcommand*{\Frefdefname}{\definitionname}%
    \newcommand*{\Frefremname}{\remarkname}%
    \newcommand*{\Frefexename}{\exercisename}%
    \newcommand*{\Frefexaname}{\examplename}%
    \newcommand*{\frefthmname}{\MakeLowercase{\Frefthmname}}%
    \newcommand*{\frefprnname}{\MakeLowercase{\Frefprnname}}%
    \newcommand*{\frefcorname}{\MakeLowercase{\Frefcorname}}%
    \newcommand*{\freflemname}{\MakeLowercase{\Freflemname}}%
    \newcommand*{\frefdefname}{\MakeLowercase{\Frefdefname}}%
    \newcommand*{\frefremname}{\MakeLowercase{\Frefremname}}%
    \newcommand*{\frefexename}{\MakeLowercase{\Frefexename}}%
    \newcommand*{\frefexaname}{\MakeLowercase{\Frefexaname}}%
}

\frefformat{plain}{\fancyrefthmlabelprefix}{%
  \frefthmname\fancyrefdefaultspacing#1%
}
\frefformat{plain}{\fancyrefprnlabelprefix}{%
  \frefprnname\fancyrefdefaultspacing#1%
}
\frefformat{plain}{\fancyrefcorlabelprefix}{%
  \frefcorname\fancyrefdefaultspacing#1%
}
\frefformat{plain}{\fancyreflemlabelprefix}{%
  \freflemname\fancyrefdefaultspacing#1%
}
\frefformat{plain}{\fancyrefdeflabelprefix}{%
  \frefdefname\fancyrefdefaultspacing#1%
}
\frefformat{plain}{\fancyrefremlabelprefix}{%
  \frefremname\fancyrefdefaultspacing#1%
}
\frefformat{plain}{\fancyrefexelabelprefix}{%
  \frefexename\fancyrefdefaultspacing#1%
}
\frefformat{plain}{\fancyrefexalabelprefix}{%
  \frefexaname\fancyrefdefaultspacing#1%
}

\Frefformat{plain}{\fancyrefthmlabelprefix}{%
  \Frefthmname\fancyrefdefaultspacing#1%
}
\Frefformat{plain}{\fancyrefprnlabelprefix}{%
  \Frefprnname\fancyrefdefaultspacing#1%
}
\Frefformat{plain}{\fancyrefcorlabelprefix}{%
  \Frefcorname\fancyrefdefaultspacing#1%
}
\Frefformat{plain}{\fancyreflemlabelprefix}{%
  \Freflemname\fancyrefdefaultspacing#1%
}
\Frefformat{plain}{\fancyrefdeflabelprefix}{%
  \Frefdefname\fancyrefdefaultspacing#1%
}
\Frefformat{plain}{\fancyrefremlabelprefix}{%
  \Frefremname\fancyrefdefaultspacing#1%
}
\Frefformat{plain}{\fancyrefexelabelprefix}{%
  \Frefexename\fancyrefdefaultspacing#1%
}
\Frefformat{plain}{\fancyrefexalabelprefix}{%
  \Frefexaname\fancyrefdefaultspacing#1%
}

\usepackage[unicode]{hyperref} 
\hypersetup{colorlinks, citecolor=Green, urlcolor=Blue, linkcolor=Red}

\DeclareMathOperator{\LieSp}{Sp}

\DeclareMathOperator{\Liegl}{\mathfrak{gl}}

\DeclareMathOperator{\Liesp}{\mathfrak{sp}}

\DeclareMathOperator{\dexp}{dexp}

\DeclareMathOperator{\cay}{cay}

\DeclareMathOperator{\OO}{O}
\DeclareMathOperator{\Diff}{Diff}

\newcommand{\dd}{\mathrm{d}}

\newcommand{\TT}{\mathrm{T}}
\newcommand{\id}{\mathrm{id}}

\newcommand{\II}{\mathrm{I}}
\newcommand{\JJ}{\mathrm{J}}
\newcommand{\RR}{\mathbb{R}}

\DeclarePairedDelimiter{\paren}{(}{)}

\DeclarePairedDelimiterX{\innerprod}[2]{\langle}{\rangle}{#1, #2}
\DeclarePairedDelimiter{\eval}{.}{\rvert}
\newcommand{\defeq}{\coloneqq}
\newcommand{\diff}[2]{\frac{\dd #1}{\dd #2}}
\newcommand{\pdiff}[2]{\frac{\partial #1}{\partial #2}}
\newcommand{\from}{\mathpunct{:}}

\newcommand{\trans}{^\mathrm{T}}
\newcommand{\coT}{\TT^{*}\!}
\newcommand{\extdiff}{\mathrm{d}}
\newcommand{\intprod}{\mathrm{i}}

%% file: topdefs.tex
\newtheoremstyle{thmstyle} 
	{\topsep}              
	{\topsep}              
	{\itshape}             
	{}                     
	{\bfseries\sffamily}   
	{.}                    
	{.5em}                 
	{\thmname{#1}\thmnumber{ #2}\thmnote{ (#3)}} 
\theoremstyle{thmstyle}
\newtheorem{theorem}{\theoremname}[section]
\newtheorem{proposition}[theorem]{\propositionname}
\newtheorem{corollary}[theorem]{\corollaryname}
\newtheorem{lemma}[theorem]{\lemmaname}

\newtheoremstyle{defstyle}
	{\topsep}
	{\topsep}
	{}
	{}
	{\bfseries\sffamily}
	{.}
	{.5em}
	{\thmname{#1}\thmnumber{ #2}\thmnote{ (#3)}}
\theoremstyle{defstyle}
\newtheorem{definition}[theorem]{\definitionname}
\newtheorem{remark}[theorem]{\remarkname}
\newtheorem{example}[theorem]{\examplename}
\newtheorem{exercise}{\exercisename}[section]

%% file: introduction.tex

An important property of a Hamiltonian system is that its flow is a symplectic map.
The idea of devising numerical methods which are themselves symplectic maps goes back into the previous century, some early references are \parencite{ruth83,feng87}.
The monographs by \textcite{hairer06,leimkuhler04} may be consulted for an extensive treatment.
Such numerical methods are called symplectic integrators and their success is often explained through the well known fact that any symplectic map can be identified as the exact flow of a local, perturbed Hamiltonian problem.
This ensures good long time behaviour in the sense that the exact Hamiltonian is approximately conserved over exponentially long times and that the numerical approximation also nearly preserves invariant tori of the exact flow.
Methods which do not possess this symplectic property will often exhibit a drift in the energy and even their global accuracy will typically deteriorate faster over long times than symplectic schemes.

As discussed in \parencite{asorey83}, a particularly attractive feature of the Hamiltonian formulation of mechanics compared to its Lagrangian counterpart is that the former distinguishes between the geometry of the problem represented by a symplectic structure and the dynamical aspects which are represented by the Hamiltonian function.
In the Lagrangian formulation this feature is absent since the symplectic structure is partly encoded in the Lagrangian function.
Turning now to time-dependent systems, we assume that the dependent variables belong to some cotangent bundle~$\coT Q$.
The usual Hamiltonian description introduces a contact structure on the space~$\coT Q \times \RR$.
By definition, this structure depends on the time-dependent Hamiltonian~$H(q, p, t)$, and thus the separation between the geometry and dynamics is again lost.
The geometric meaning of a canonical transformation is therefore no longer clear as in the autonomous case.
A common approach is to extend the system by adding an extra position variable.
This variable can be interpreted as a new time variable.
Then one may consider the extended phase space $\coT (Q \times \RR)$ which can be furnished with a symplectic form, see for instance \parencite{struckmeier05}.

In the late 1990s a renewed interest in the numerical solution of linear non-autonomous differential equations was sparked, in particular through some pioneering papers by Iserles and Nørsett, see e.g.~\parencite{iserles99}, where they developed numerical methods based on the Magnus expansion~\parencite{magnus54}.
There are also other similar ways of representing the exact flow of such problems, for instance the Fer expansion~\parencite{fer58}, see also \parencite{iserles84} and \parencite{blanes98}.
This activity resulted in several new contributions to the numerical solution of linear and quasi-linear non-autonomous PDEs, see e.g.~\parencite{hochbruck03,gonzalez06,gonzalez07}.
Another application branch of such methods is highly oscillatory linear non-autonomous ODEs.
Asymptotic analysis can be used to show excellent behaviour of the global error when the dominating frequencies of the problem tend to infinity, see for instance \parencite{grimm06} and \parencite{iserles02}.

In this paper, we attempt to present a more geometric view on integrators for non-autono\-mous systems, and we give particular attention to methods which have an exponential character, such as Magnus integrators.
We use the definition of canonical transformations introduced by \textcite{asorey83}.
Their framework is relatively general and we shall consider the question of which numerical integrators can be characterized as canonical transformations.
In particular we shall see that the most common exponential integrators for non-autonomous linear problems can be furnished with such a property.
Finally, we provide numerical evidence showing that canonicity in this sense together with symmetry of the scheme appear to be important for the long term behaviour of integrators.
It is well known from the literature that if such methods are also exponential, they can have excellent properties, as is for instance the case for Magnus integrators.
However, being exponential without any of these two additional properties will typically not yield a good approximation of the Hamiltonian over long times.

%% file: chap2.tex

In this section we consider the four possible combinations of autonomous and non-autono\-mous, linear and non-linear differential equations.

\paragraph*{Autonomous, linear (AL) problems.}

The AL case can be written as
\[
	\dot y = A y, \quad y(0) = y_0,
\]
where $A \in \RR^{d \times d}$ is constant.
The solution to AL problems can be represented exactly by means of the matrix exponential,
\[
	y(t) = \exp(t A) y_0,
\]
thus, numerical methods for this class amount to considering methods of computing or approximating the matrix exponential, see e.g.~\parencite{moler03}.
We will not consider AL problems in this paper.

\paragraph*{Autonomous, non-linear (AN) problems.}
The AN case can be written as
\[
	\dot y = f(y), \quad y(0) = y_0,
\]
where $f \from \RR^d \to \RR^d$.
Most numerical schemes for ordinary differential equations are conveniently applied to problems written in this format and are treated in several monographs and textbooks such as \parencite{hairer93}.

\paragraph*{Non-autonomous, linear (NL) problems.}

The NL case can be written as
\[
	\dot y = A(t) y, \quad y(0) = y_0,
\]
where $A \from \RR \to \RR^{d \times d}$.
Since this problem class constitutes a subset of the non-linear problems, most general numerical schemes for ODEs can be applied also to this class.
However, there exist several classes of integrators which are tailored for this problem type, two of which are the Magnus methods \parencite{iserles99} and methods based on the Fer expansion \parencite{fer58}.
In particular, such methods have found applications to non-autonomous linear PDEs such as the time-dependent Schrödinger equations \parencite{hochbruck03} and to highly oscillatory problems, see~\parencite{iserles02,khanamiryan12} and the references therein.

We can turn NL problems into AN problems by substituting $t$ with a new variable~$y^{d + 1}$ and appending the ODE~$\dot y^{d + 1} = 1$.
This process is called \emph{autonomization}.
By doing this, we are replacing a linear problem by a non-linear problem, which may be more difficult to solve numerically.
NL problems are the main focus of this paper.

\paragraph*{Non-autonomous, non-linear (NN) problems.}

The NN case can be written as
\[
	\dot y = f(y, t), \quad y(0) = y_0,
\]
where $f \from \RR^d \times \RR \to \RR^d$.
NN problems can be turned into AN problems by autonomization.
This class of problems is not the main focus in this paper.

%% file: conclusion.tex
In this paper we have taken a new look at numerical integrators for Hamiltonian problems where the energy function depends explicitly on time.
Using the framework of canonical transformations defined by \parencite{asorey83}, we have characterized integrators which are canonical according to this definition.
In particular we have studied methods for linear non-autonomous equations, a problem class which has attracted considerable interest from the numerical analysis community in recent decades.
We have not obtained analytical results which rigorously support  the hypothesis that canonical methods can be expected to have good long time behaviour.
However, numerical tests for a toy problem, a smooth oscillator, seem to corroborate such an assumption.
It is unclear whether the by now classical approach of backward error analysis will be a useful tool in studying error growth of canonical methods since the analysis should allow for highly oscillatory problems and linear PDEs.
We believe however, that the notion of canonical transformations used in this paper may be a viable route to gain a better insight into the excellent properties of exponential integrators applied to linear non-autonomous Hamiltonian problems.